\documentclass[a4paper,12pt]{amsart}
\usepackage{amsfonts,amssymb}

\usepackage{mathrsfs}
\let\mathcal\mathscr

\usepackage[colorlinks=true,citecolor=blue, urlcolor=blue, breaklinks, linkcolor=black]{hyperref}

\usepackage{array}
\usepackage[left=2.3cm,top=3.3cm,right=2.3cm,bottom=2.7cm]{geometry}
\usepackage[all,ps,cmtip]{xy} 

\def\Z{{\mathbf Z}}

\def\C{{\mathbf C}}
\def\CP{{\mathbf{CP}}}

\def\R{{\mathbf R}}

\def\P{{\mathbf P}}

\def\B{{\bf B}}
\def\bS{{\mathbf S}}

\def\cha{characteristic}

\def\phi{{\varphi}}

\def\cO{\mathcal{O}}

\def\llra{\hbox to 10mm{\rightarrowfill}}

\def\lllra{\hbox to 15mm{\rightarrowfill}}

\def\isom{\simeq}

\def\ie{\hbox{i.e.}}

\DeclareMathOperator{\Td}{td}
\DeclareMathOperator{\td}{td}

\DeclareMathOperator{\ch}{ch}

\DeclareMathOperator{\Pic}{Pic}

\def\llra{\hbox to 10mm{\rightarrowfill}}
\def\lllra{\hbox to 15mm{\rightarrowfill}}

\newtheorem{lemm}{Lemma}
\newtheorem{theo}[lemm]{Theorem}
\newtheorem{coro}[lemm]{Corollary}

\theoremstyle{definition}

\newtheorem*{ques*}{Question}

\theoremstyle{remark}
\newtheorem*{remark*}{Remark}
\newtheorem*{note*}{Note}

\begin{document}
\title{Cohomological characterizations of the complex projective space}

\author[O.~Debarre]{Olivier Debarre}
\address{Univ  Paris Diderot, \'Ecole normale su\-p\'e\-rieu\-re, PSL Research University, CNRS, D\'epar\-te\-ment Math\'ematiques et Applications\\45 rue d'Ulm, 75230 Paris cedex 05, France}
\email{{olivier.debarre@ens.fr}}
\urladdr{\url{http://www.math.ens.fr/~debarre/}}

\begin{abstract}
In this survey, we discuss whether  the complex projective space can be characterized by its integral cohomology ring among compact complex manifolds.
\end{abstract}


\maketitle


\section{Introduction}

Our starting point is the following 1957 result from \cite{hk}.

\begin{theo}[Hirzebruch--Kodaira, Yau]
Any compact  K\"ahler manifold which is homeomorphic to $\CP^n$ is  biholomorphic to $\CP^n$.
\end{theo}

Actually, Hirzebruch--Kodaira proved this result  under the additional assumptions that the  K\"ahler manifold   is {\em diffeomorphic} to $\CP^n$ and that, when $n$ is even,   $c_1(T_X)$ is not  $n+1$ times a negative generator of $H^2(X,\Z)$. The first assumption was    dropped (see \cite{mor}) when Novikov    proved in \cite{nov} that Pontryagin classes are invariant under homeomorphisms; the second assumption was also   dropped later thanks to   work of Yau (\cite{yau}). If one does not assume that $X$ is K\"ahler, the conclusion still holds for $n\le 2$ (complex surfaces with even first Betti number  are K\"ahler by \cite{buc,lam}), but nothing is known for $n\ge 3$; if the  K\"ahler assumption can be dropped when $n=3$,  the sphere $\bS^6$ has no complex structure.

Stronger characterizations were proved in   dimensions $n\le 6$ by Fujita (\cite{fuj1}) and Libgober--Wood   (\cite{lw}) assuming only that  the K\"ahler manifold has the homotopy type of $\CP^n$. Looking carefully through their arguments, it is not too difficult to extract a proof of the following stronger result.

\begin{theo}
Let  $n$ be an integer with $n\le 6$. Any  compact 
K\"ahler manifold with the same  integral cohomology ring as  $\CP^n$ is 
\begin{itemize}
\item either  isomorphic to $\CP^n$;
\item or a quotient of the unit balls $\B^4$ or $\B^6$.
\end{itemize}
\end{theo}

No quotients of even dimensional unit balls $\B^{2m}$ with the same  integral cohomology rings as  $\P^{2m}$ are known  (all known examples have torsion in $H^2$). It is therefore legitimate to ask the following question.

\begin{ques*} Is any compact K\"ahler  manifold with same integral cohomology ring as $\CP^n$ isomorphic to $\CP^n$? 
\end{ques*}

 The methods used in the   proof of the theorem above are completely computational and it seems unlikely that they can be generalized to higher dimensions (we obtain only partial results in dimension 7 in Theorem \ref{th7}).  Using geometrical arguments would perhaps be a good idea to make further progress.

\subsection*{Acknowledgements}
All the computations were done with the  software {\tt Sage} (\cite{sage}). Many thanks to Pierre Guillot for his computations of the polynomials $t_n$ in Section \ref{seco}.

\section{Preliminaries}

From now on, we will write $\P^n$ instead of $\CP^n$ for the complex projective space of dimension $n$.

\subsection{Hirzebruch--Riemann--Roch}\label{se1.1} Let $X$ be a  projective complex  manifold  of dimension $n$. Following \cite{hir}, we set
 \begin{equation*}
\chi^p(X):=\sum_{q=0}^n(-1)^qh^{p,q}(X)=\chi(X,\Omega_X^p) 
\end{equation*}
and we define  the $\chi_y$-genus
 \begin{equation*}
\chi_y(X):=\sum_{p=0}^n  \chi^p(X) y^p=\sum_{p, q=0}^n(-1)^qh^{p,q}(X)y^p \in \Z[y].
\end{equation*}
 For instance, $\chi_0(X)=\chi(X,\cO_X )$ and $\chi_{-1}(X)=\chi_{\rm top}(X )$. One consequence of the Hirzebruch--Riemann--Roch theorem is that the coefficients $\chi^p(X)$ of the polynomial $\chi_y(X)$ can be expressed in terms of the Chern classes of $X $ (\cite[Section~IV.21.3,~(10)]{hir})
 \begin{equation*}
\chi^p(X)=T_n^p(c_1(X),\dots,c_n(X)) \quad\text{or}\quad  \chi_y(X)=T_n(y;c_1(X),\dots,c_n(X)),
\end{equation*}
 where  $T_n(y;c_1,\dots,c_n):= \sum_{p=0}^nT_n^p(c_1,\dots,c_n)y^p$. These polynomials satisfy $T^p_n=(-1)^nT_n^{n-p}$ and they can be explicitly determined (\cite[Section~I.1.8,~(10)]{hir}).
For example, the constant terms $T_n^0(c_1,\dots,c_n)$ (which are also $(-1)^n$ times the  leading term) are the Todd polynomials   $\td_n(c_1,\dots,c_n)$  (\cite[Section~I.1.7,~(10)]{hir})
and $T_n (-1;c_1,\dots,c_n)=c_n$, so that
\begin{equation}\label{cn0}
c_n(X)=\chi_{\rm top}(X) .
\end{equation}
Libgober--Wood also introduce the polynomials 
\begin{equation}\label{tn}
t_n(z;c_1,\dots,c_n):=T_n(z-1;c_1,\dots,c_n) .
\end{equation}
They show (\cite[Lemma~2.2]{lw})
\begin{equation}\label{tnn}
t_n(z;c_1,\dots,c_n) =c_n-\tfrac12nc_nz+\tfrac{1}{12}\bigl( \tfrac12 n(3n-5)c_n+c_1c_{n-1}\bigr)z^2 +\cdots
\end{equation}
 and   they compute  these polynomials for $n\le 6$ (\cite[p.\ 145]{lw}). We extend their computations to all $n\le9$ in Section~\ref{seco}.

\subsection{Compact K\"ahler  manifolds with same
Betti numbers as $\P^n$} Let $X$ be a    compact K\"ahler  manifold  with the same
Betti numbers as $\P^n$.

Since $X$ is K\"ahler, one can compute the numbers $h^{p,q}(X)$ from its Betti numbers, and we see that $h^{p,q}(X)= h^{p,q}(\P^n)=1$ if    $p=q\in\{0,\dots,n\}$, and $h^{p,q}(X)=0$ otherwise.
 In particular, $X$ is projective (Kodaira). 
Setting $c_i(X):=c_i(T_X)\in H^{2i}(X,\Z)$, we deduce from the Hirzebruch--Riemann--Roch theorem (Section~\ref{se1.1}) the equalities
\begin{multline}\label{tnpn}
\qquad t_n(z;c_1(X),\dots,c_n(X))=t_n(z;c_1(\P^n),\dots,c_n(\P^n)) \\{}=t_n\bigl(z;\tbinom{n+1}{1},\dots,\tbinom{n+1}{n}\bigr)=\sum_{i=0}^n\tbinom{n+1}{i+1}(-1)^iz^i.\qquad
\end{multline}
In particular, \eqref{tnn} implies
 \begin{equation}\label{cn} 
c_n(X)=c_n(\P^n)=n+1\ ,\ c_1(X)c_{n-1}(X)=c_1(\P^n)c_{n-1}(\P^n)=\tfrac12n(n+1)^2.
\end{equation}

Assume now $c_1(X)<0$,\footnote{In the sense that the image of $c_1(X)$ in $H^2(X,\R)$ is a negative multiple of the class of a K\"ahler metric.}   so that $X$ is of general type. We have (\cite[Remark (iii)]{yau})
\begin{equation}\label{yauu}
\begin{array}{c}
 \bigl(\tfrac{2(n+1)}{n}c_2(X)-c_1(X)^2\bigr)\cdot (-c_1(X))^{n-2}\ge 0 \\  
 \hbox{with equality if and only if $X$ is covered by the unit ball in $\C^n$,}
 \end{array}
\end{equation}
in which case, by the Hirzebruch proportionality principle, all the  Chern numbers of $X$ are the same as those of $\P^n$ and $n$ is even.

If on the other hand  $c_1(X)>0$, so that $X$ is a Fano manifold, the group $\Pic(X)\isom H^2(X,\Z)$ is torsion-free (\cite[Proposition 2.1.2]{ip}) and 
we write $K_X =-c_1L$, where $L$ is an ample generator of $\Pic(X)$. We have (\cite{ko})
\begin{equation}\label{koo}
c_1\le n+1 \quad 
 \hbox{with equality if and only if $X\isom \P^n$.}
\end{equation}

\subsection{Compact K\"ahler  manifolds with same
  integral cohomology ring as $\P^n$}

Assume   now that  $X$ has the same
  integral cohomology ring as $\P^n$.  We have $\Pic(X)=\Z L$, where $L$ is ample with $L^n=1$, and $\ell:=c_1(L)$ generates $H^2(X,\Z) $. We define integers $c_1,\dots,c_n$ by setting $c_i(X)=c_i\ell^i$ and we  compute   Euler \cha s using the Hirzebruch--Riemann--Roch theorem (\cite[Theorem~20.3.2]{hir})
\begin{eqnarray}
\chi(X,L^m)&=&\bigl[e^{m\ell}\cdot \Td(X)\bigr]_n = \Td_n(X)+\cdots +\frac{m^{n-2}}{(n-2)!}\Td_2(X)
\label{holt}\\
&&\hskip55mm{}+\frac{m^{n-1}}{(n-1)!}\Td_1(X)+
\frac{m^n}{n!},\nonumber\\
 \chi(X,T_X\otimes L^m)&=&\bigl[\ch(T_X)\cdot e^{m\ell}\cdot \Td(X)\bigr]_n.
\label{tholt}\end{eqnarray}

\begin{lemm}\label{l1}
We have
\begin{eqnarray}
c_1-(n+1)&\equiv& 0\pmod2,\label{c1}\\
c_1^2+c_2-3nc_1+\tfrac12(n+1)(3n-2)&\equiv& 0\pmod{12}.\label{c1c2}
\end{eqnarray}
\end{lemm}

\begin{proof}
Since the     polynomial in $m$ which appears in \eqref{holt} takes integral values at all integers $m$, it decomposes as
$$\binom{m+n}{n}+a_1\binom{m+n-1}{n-1}+a_2\binom{m+n-2}{n-2}+\cdots$$
where $a_1,a_2,\dots$ are integers. The coefficient of $m^{n-1}$ is
$$\frac1{n!}\sum_{i=1}^ni+\frac1{(n-1)!}a_1=\frac1{(n-1)!}\Td_1(X)=\frac1{(n-1)!}\frac12c_1,$$
hence $\frac{n+1}{2}+a_1=\frac12c_1$. This proves the   congruence \eqref{c1}.

The coefficient of $m^{n-2}$ is
$$\frac1{n!}\sum_{1\le i<j\le n}ij+\frac1{(n-1)!}\sum_{i=1}^{n-1}ia_1+\frac1{(n-2)!}a_2=\frac1{(n-2)!}\Td_2(X)=\frac1{(n-2)!}\frac1{12}(c_1^2+c_2).$$
The first sum is
$$\sum_{1\le i<j\le n}\! ij=\sum_{1\le j\le n}\! \frac{j(j-1)}{2}j=\frac12\Big[ \frac{n^2(n+1)^2}{4}-\frac{n(n+1)(2n+1)}{6}\Bigr]=\frac{n(n-1)(n+1)(3n+2)}{24}.$$
We obtain
$$\frac1{12}(c_1^2+c_2)=\frac{(n+1)(3n+2)}{24}+\frac{n}{2}\Bigl(  \frac12c_1-\frac{n+1}{2}\Bigr)+a_2
$$
and the   congruence \eqref{c1c2} follows.
 \end{proof}

 \section{Surfaces}
 
 \begin{theo}[\cite{yau}] 
Any  compact 
complex manifold with the same Betti numbers as  $\P^2$ is 
\begin{itemize}
\item either  isomorphic to $\P^2$;
\item or a quotient of the unit ball $\B^2$.
\end{itemize}
\end{theo}

\begin{proof} 
A  compact 
complex surface with even first Betti number is K\"ahler (\cite{buc,lam}). Equations \eqref{cn} then give $c_1(X)^2=9 $ and $c_2(X)=3$.
 
 If $c_1(X)>0$, the surface $X$ is    isomorphic to $\P^2$ by \eqref{koo}.
 
 If $c_1(X)<0$, there is equality in \eqref{yauu} and $X$ is a quotient of $\B^2$.
 \end{proof} 
 
  \begin{coro} 
Any  compact 
complex manifold with the same integral cohomology groups as  $\P^2$ is 
   isomorphic to $\P^2$.
 \end{coro}
 
 \begin{proof}
Compact quotients $X$ of the unit ball $\B^2$ are called fake projective planes. They are all classified and it was proved in  \cite[Theorem~10.1]{py} that $H_1(X,\Z)$ is always nonzero (and torsion). It follows that $H^2(X,\Z)_{\rm tors}\isom H_1(X,\Z)_{\rm tors}$ is nonzero, so the integral cohomology groups of fake projective planes are different  from those of  $\P^2$.
\end{proof}
 
 \section{Threefolds}
 
 In odd dimensions $2m-1$, it is definitely not enough to assume that the Betti numbers of $X$ and $\P^{2m-1}$ are the same: a smooth odd-dimensional quadric $X\subset \P^{2m}$ has this property, but, if $m\ge 2$, a positive generator $L$ of $H^2(X,\Z)$ satisfies $L^{2m-1}= 2$, hence $X$ is not even homeomorphic to $\P^{2m-1}$. In dimension 3, there are two other examples of Fano threefolds with the same Betti numbers as $\P^3$ (this is equivalent in that case to $b_2=1$ and $h^{1,2}=0$):  one with $L^3=5$ and one with $L^3=22$ (\cite[Table~12.2]{ip}). 
 
  \begin{theo}[\cite{fuj1,lw}] 
Any  compact 
K\"ahler manifold with the same  integral cohomology ring as  $\P^3$ is 
  isomorphic to $\P^3$.\end{theo}

\begin{proof} 
  If $\ell$ is a positive generator of $H^2(X,\Z)$, we write as before  $c_i(X)=c_i\ell^i$.
 Equations  \eqref{cn}    give $c_3=4$ and  $c_1  c_2 =24$.
 If $c_1 <0$, we get $c_2 <0$, but this contradicts   \eqref{yauu}. Therefore, $c_1 $ is a positive divisor of $24$ which we can assume, by \eqref{koo}, to be 1, 2,   3, or 4. By Lemma \ref{l1}, $c_1$ is even, so we need only exclude $c_1=2$. In that case, $(X,L)$ is a so-called del Pezzo 
variety (coindex 2) with $L^3=1$. It is therefore isomorphic to a hypersurface of degree 6 in the weighted projective space $\P(3,2,1,1,1)$ (\cite{fuj}, \cite[Theorem 3.2.5]{ip}). But such a variety has $h^{1,2}=21$, so this is a contradiction. \end{proof}

  \section{Fourfolds}
  
    \begin{theo}[\cite{fuj1,lw}] 
Any  compact 
K\"ahler manifold with the same  integral cohomology ring as  $\P^4$ is 
\begin{itemize}
\item either  isomorphic to $\P^4$;
\item or a quotient of the unit ball $\B^4$.
\end{itemize}\end{theo}

 Four examples  of compact quotients $X$ of $\B^4$ with the same Betti numbers as $\P^4$ are known, but  the groups $H_1(X,\Z)$ are never zero (\cite[Theorem 4]{py2}) hence they do not have the same  integral cohomology ring as  $\P^4$. It is therefore possible that the second case of the theorem never occurs.

\begin{proof} 
 If $\ell$ is a positive generator of $H^2(X,\Z)$, we write as before  $c_i(X)=c_i\ell^i$.
  Equations  \eqref{cn}    give $c_4=5$ and $c_1  c_3 =50$. By Lemma \ref{l1}, $c_1$ is odd, so by \eqref{koo}, we are reduced to $c_1\in\{ \pm 1,\pm 5, -25\}$.
  
  Equation  \eqref{cn}    reads $c_4 =5 $.   Equation \eqref{holt} with $m=0$ then  gives (using the values for the Todd polynomials given  in  Section~\ref{seco}) 
  $$1=\chi(X,\cO_X)=\tfrac{1}{720}\bigl(-c_1^4 + 4c_1^2 c_2+ 3 c_2^2 +  c_1 c_3-  c_4 \bigr)=\tfrac{1}{720}\bigl( -c_1^4 + 4c_1^2 c_2+ 3 c_2^2 +   50-5 \bigr)
  $$
  so $c_2$ is an integral root of the quadratic equation 
  $$3c_2^2+4c_1^2c_2-c_1^4-675=0.$$
 Its (reduced) discriminant $ 7c_1^4+2025 $ is therefore a square, which leaves only (among our possible values) the  cases $c_1=\pm 5$, $c_2=10$.
  
If $c_1=5$, the fourfold $X$ is isomorphic to $\P^4$ by \eqref{koo}. If $c_1=-5$, there is equality in \eqref{yauu}   and $X$ is 
a quotient of $\B^4$.
\end{proof} 
  
   \section{Fivefolds}
   
       \begin{theo} [\cite{fuj1,lw}] 
Any  compact 
K\"ahler manifold with the same  integral cohomology ring as  $\P^5$ is 
  isomorphic to $\P^5$.
 \end{theo}

\begin{proof} 
 If $\ell$ is a positive generator of $H^2(X,\Z)$, we write as before  $c_i(X)=c_i\ell^i$.
  Equations  \eqref{cn}    give $c_5=6$ and  $c_1  c_4 =90$ and, by   \eqref{c1}, $c_1$ is even. 
  The relation $\chi(X,\cO_X)=T^0_5(c_1,\dots,c_5)$ (Section~\ref{se1.1}) gives
  $$1=\chi(X,\cO_X)=-\tfrac{1}{1440}\bigl(c_1^3  c_2 - 3 c_1  c_2^2 -  c_1^2  c_3 + c_1  c_4\bigr)=-\tfrac{1}{1440}\bigl( c_1^3  c_2 - 3 c_1  c_2^2 -  c_1^2  c_3 + 90 \bigr)
  $$
  so $c_2$ is an integral root of the quadratic equation 
\begin{equation}\label{quad}
3c_1c_2^2-c_1^3c_2+c_1^2c_3-1530=0.
\end{equation}
  
  Moreover, 
the   congruence \eqref{c1c2} reads
\begin{equation}\label{cong}
 c_1^2-3c_1+c_2+3\equiv 0\pmod{12}.
\end{equation}

 If $c_1\equiv 0\pmod9$, we obtain, reducing  \eqref{quad} modulo 27, the contradiction $1530\equiv 0\pmod{27}$. Using \eqref{koo}, it follows that the possible values for $c_1$ are $\pm 2, \pm 6  , - 10,  - 30$.
 
  We rules these cases out one by one.

  \medskip
    \noindent{\bf Case $c_1=-30$.} Reducing \eqref{quad} modulo 25, we obtain $3\cdot (-30) c_2^2\equiv 1530\equiv 5\pmod{25}$, hence 
    $1\equiv 3\cdot (-6) c_2^2\equiv 2 c_2^2 \pmod{5}$, which is impossible.
   
   \medskip 
 \noindent{\bf Case $c_1=-10$.} Reducing \eqref{quad} modulo 25, we obtain $3\cdot (-10) c_2^2\equiv 1530\equiv 5\pmod{25}$, hence 
    $ c_2^2 \equiv-1\pmod{5}$. We have $c_4=-9$. Equation
   \eqref{holt} reads, for $m=1$,
\begin{eqnarray*}
720\chi(X,  L)&=&720+   (-c_1^4 + 4c_1^2 c_2+ 3 c_2^2 +  c_1 c_3-  c_4) + 15c_1c_2+ 10(c_1^2+c_2 )+ 15 c_1+ 6\\
&\equiv& 9 +3\cdot (-1)+ 6\pmod5,
 \end{eqnarray*}
which is absurd.
    
    \medskip      
   \noindent{\bf Case $c_1=\pm 2$.}  By \eqref{cong}, we can write $c_2=12d_2-1$.

Assume $c_1=2$.  Substituting $c_5=6$, $c_4=45$, and   $c_3=(1530-6c_2^2+8c_2)/4$ (obtained from \eqref{quad}) in \eqref{tholt} with $m=-1$, we compute, following \cite{fuj1},
   $$
24\chi(X,T_X\otimes L^{-1})= -3 c_2^2 + 2c_2 + 731=-3 (12 d_2-1)^2 + 2(12 d_2-1) + 731
 \equiv 6\pmod{24},
   $$
   which is absurd. When $c_1=-2$, we obtain similarly the contradiction
$$
24  \chi(X,T_X\otimes L)=3 c_2^2 - 2c_2 + 727\equiv 12\pmod{24}.
$$
 
  \medskip 
    \noindent{\bf Case $c_1=-6$.} By \eqref{cong}, we can write $c_2=12 d_2+3$. 
 We compute, using \eqref{quad} and \eqref{tholt} with $m=-1$ again,
   \begin{eqnarray*}
24  \chi(X,T_X\otimes L )&=& 45c_2^2 - \tfrac{520}{3}c_2 - 171=  45(12 d_2+3)^2 - \tfrac{520}{3}(12 d_2+3) - 171\\
 &\equiv& 45\cdot 9-520\cdot 4  d_2-520-  171 \pmod{24},
   \end{eqnarray*}
but this is absurd since this last number is $\equiv  2 \pmod8$.
%
\end{proof}

      \section{Sixfolds}

       \begin{theo}[\cite{lw}] 
Any  compact 
K\"ahler manifold with the same  integral cohomology ring as  $\P^6$ is 
\begin{itemize}
\item either  isomorphic to $\P^6$;
\item or a quotient of the unit ball $\B^6$.
\end{itemize}\end{theo}

\begin{proof} 
 If $\ell$ is a positive generator of $H^2(X,\Z)$, we write as before  $c_i(X)=c_i\ell^i$.
  Equations \eqref{cn}    gives $c_6=7$ and $c_1  c_5 =3\cdot 7^2=147$. From the fact that the polynomial $t_6(y;c_1,\dots,c_6)$ is the same for $X$ and $\P^6$, we obtain
  \begin{eqnarray*}
720\tbinom{7}{5}\hskip-2mm&=&\hskip-2mm -  c_1^3  c_3 + 3  c_1  c_2  c_3 +   c_1^2  c_4 - 3 c_3^2 + 3  c_2  c_4 + 69 c_1  c_5 + 186  c_6,  \\
60480\tbinom{7}{7}\hskip-2mm&=&\hskip-2mm 2 c_1^6 - 12  c_1^4  c_2 +  11 c_1^2  c_2^2 + 5  c_1^3  c_3 + 10  c_2^3 +  11  c_1  c_2  c_3 - 5  c_1^2  c_4 -    c_3^2 - 9c_2  c_4 - 2  c_1  c_5 + 2c_6.
\end{eqnarray*}  
Plugging in the values $c_6=7$ and $c_1  c_5 = 147$, we obtain (\cite[p.\ 150]{lw})
  \begin{eqnarray}
  -  c_1^3  c_3 + 3  c_1  c_2  c_3 +   c_1^2  c_4 - 3 c_3^2 + 3  c_2  c_4   &=& 3675, \label{0}\\
2 c_1^6 - 12  c_1^4  c_2 +  11 c_1^2  c_2^2 + 5  c_1^3  c_3 + 10  c_2^3 +  11  c_1  c_2  c_3 - 5  c_1^2  c_4 -    c_3^2 - 9c_2  c_4  &=& 60760.\label{1}
\end{eqnarray} 
Eliminating $c_4$ between these two equations, we see that $c_3$ is a solution of the quadratic equation
\begin{equation}\label{quadeq}
a_2c_3^2+a_1c_3+a_0=0,
\end{equation}
where 
\begin{eqnarray*}
a_2&=&2(15+8c_1^2),\\
a_1&=& -4c_1c_2(15c_2+8c_1^2),\\
a_0&=&-30c_2^2-43c_1^2c_2^3+25c_1^4c_2^2+6c_1^6c_2+215355c_2-2c_1^8+79135c_1^2.
\end{eqnarray*}
In particular, $15c_2+8c_1^2$ divides $a_0$, hence also the remainder of the  division of $1125a_0$ by $15c_2+8c_1^2$, which is
$$R(c_1):=c_1^2(6758c_1^6-40186125).
$$

Moreover, 
Equation \eqref{holt}    gives
that the   polynomial 
\begin{eqnarray*}
P(m)&=&m\frac{1}{1440}\bigl( -c_1c_4+c_1^2c_3+3c_1c_2^2 -c_1^3c_2 \bigr) +  \frac{m^2}{2}\frac{1}{720}\Bigl(-c_4+c_1c_3+3c_2^2+4c_1^2c_2-c_1^4\Bigr)\\
&& \hskip 1cm {}+ \frac{m^3}{6}\cdot\frac{1}{24}c_1c_2+ \frac{t^4}{24}\cdot\frac{1}{12}\Bigl(c_1^2+c_2\Bigr)+ \frac{m^5}{120}\cdot\frac{1}{2}c_1+ \frac{m^6}{720} 
 \end{eqnarray*}
 takes integral values for all integers $m$. In particular, 
\begin{equation}\label{+}
720(P(1)+P(-1))=-c_4+c_1c_3+3c_2^2+4c_1^2c_2-c_1^4  + 5   (c_1^2+c_2 )+2  \equiv 0\pmod{720}.
\end{equation}
Finally, by \eqref{c1c2}, we have the congruence
\begin{equation}\label{cog}
   c_1^2+c_2+6c_1 \equiv 4\pmod{12}.
\end{equation}

   \medskip
\noindent{\bf Case     $  c_1\equiv 0\pmod3$. } We get $c_2  \equiv 1\pmod3$ (from \eqref{cog}), 
$    c_4  \equiv 1\pmod{3}$ (from \eqref{+}), and $  c_3 \equiv 0\pmod{3}$ (from \eqref{0}).

 We write $c_1=3d_1$, $c_2=3d_2+1$, and $c_4=3d_4+1$.  We have 
$  -c_4 + 3 +5c_2  +2\equiv0\pmod{9}$ (from \eqref{+} again), \ie,  $ d_2+ d_4  \equiv  0 \pmod{3}$.  
Moreover,
using \eqref{0} modulo 27, we obtain
$9d_1^2(3d_4+1)
+3(3d_2+1)(3d_4+1) \equiv 3\pmod{27}
$ (from \eqref{0} again), \ie,
 $d_4+d_2 + d_1^2 \equiv 0\pmod{3}
$. This gives $d_1 \equiv 0\pmod{3}$, which is impossible, since $d_1$ is a power of $7$.
\medskip

The case $c_1=49$ being excluded by \eqref{koo}, there remains to consider the cases  $c_1=-49$,  $c_1=\pm1$, and $c_1=\pm7$. In all these cases, we have $c_2  \equiv -3\pmod{12}$ by \eqref{cog} and we write $c_2=12e_2-3$.  We saw above that  $R(c_1) $ must be divisible by $15c_2+8c_1^2=180e_2+8c_1^2-45$.

   \medskip
\noindent{\bf Case  $c_1=-49$.} The integer  $d=15c_2+8c_1^2=180e_2-45+8\cdot 49^2$ is   positive by \eqref{yauu} and divides
$R(c_1)= 7^6\cdot 37\cdot 251\cdot 1559\cdot 131849$.  A  computer check gives us  all the   positive divisors $d$ of $R(c_1)$ for which   $e_2$ is an integer;  we then compute the discriminant $D$ of the quadratic equation \eqref{quadeq}. We find   
\begin{itemize}
\item $d=37\cdot 1559$, for which $e_2=214$ and $D=2^7 \cdot 37 \cdot 1559 \cdot 7087681 \cdot 21780337$;
\item $d=7\cdot 131849$, for which $e_2=5021$  and $D= 2^8 \cdot 7^3 \cdot 13 \cdot 193 \cdot 131849 \cdot 9694436995073$;
\item $d=7^3\cdot 251\cdot 1559\cdot 131849 $, for which $e_2=98314662210$ and $D= 2^7 \cdot 7^8 \cdot 17 \cdot 251 \cdot 317 \cdot 1559 \cdot 131849 \cdot 165057229 \cdot
1203263426047496730660859$.
\end{itemize}
 In each case,   $D$  is not a perfect square hence the system of equations \eqref{0} and \eqref{1} has no integral solutions.

   \medskip
\noindent{\bf Case   $c_1=\pm1$.} We have $R(c_1)=-23 \cdot 1746929$ and 
there are no divisors of $R(c_1)$ for which   $e_2$ is an integer.

   \medskip
\noindent{\bf Case $c_1= 7$.}  There is then   equality in \eqref{koo}   and $X$ is 
isomorphic to $\P^6$.

   \medskip
\noindent{\bf Case $c_1=-7$.} We have $R(c_1)=7^4 \cdot 101 \cdot 152533$. The only divisor of $R(c_1)$ for which   $e_2$ is an integer is $7\cdot 101$, for which $e_2=2$ and $c_2=21$. There is then   equality in \eqref{yauu}   and $X$ is 
a quotient of $\B^6$.
   \end{proof}

 \section{Sevenfolds}
         
 \begin{theo}\label{th7}
Any  compact 
K\"ahler manifold $X$ with the same  integral cohomology ring as  $\P^7$ is 
  isomorphic to $\P^7$, unless $c_1(X)^7\in \{\pm 2^7,\pm 4^7\}$.
 \end{theo}

 \begin{proof} 
      If $\ell$ is a positive generator of $H^2(X,\Z)$, we write as before  $c_i(X)=c_i\ell^i$.
  Equations  \eqref{cn}      give  $c_7=8$ and $c_1  c_6 =  2^5\cdot 7=224$ and, by   \eqref{c1}, $c_1$ is even.  
  From the fact that the polynomial $t_7(y;c_1,\dots,c_7)$ is the same for $X$ and $\P^7$, we obtain, comparing the coefficients of $y^5$ and $y^7$ and plugging in the values $c_7=8$ and $c_1  c_6 = 224$, the equations
 \begin{eqnarray}
0 \hskip-2mm&=&\hskip-2mm   	c_1^3 c_4 - 3 c_1 c_2 c_4 - c_1^2 c_5 + 3 c_3 c_4 - 3 c_2 c_5 + 7728
, \label{75}\\
0\hskip-2mm&=&\hskip-2mm -2 c_1^5 c_2 + 10 c_1^3 c_2^2 + 2 c_1^4 c_3 - 10 c_1 c_2^3 - 11 c_1^2 c_2 c_3 - 2 c_1^3 c_4 + c_1 c_3^2 + 9 c_1 c_2 c_4 \nonumber\\
&&\hskip35mm{} + 2 c_1^2 c_5 + 120512  \label{77},
\end{eqnarray} 
with 
$7728 =2^4 \cdot 3 \cdot 7 \cdot 23 $ and $ 120512=2^6 \cdot 7 \cdot 269$.

By \eqref{c1c2}, we also have the congruence
\begin{equation}\label{cog7}
   c_1^2+c_2+3c_1 \equiv 8\pmod{12}.
\end{equation}

We also compute, for all integers $m$,
  \begin{eqnarray}
60480\chi(X,L^m)&=&12 m^7 
+ 42 m^6 c_1 
+ 42 m^5 (c_1^2 +     c_2)
+ 105 m^4 c_1 c_2\nonumber\\
&&{}
+  2m^3(-7 c_1^4 + 28   c_1^2 c_2+ 21   c_2^2 + 7   c_1 c_3 - 7   c_4)
\nonumber\\
&&{}+  m^2 (-21 c_1^3 c_2 + 63   c_1 c_2^2  + 21 c_1^2 c_3 -
21   c_1 c_4)
\nonumber\\
&&{}+  m( 2   c_1^6     - 12   c_1^4 c_2 + 11   c_1^2 c_2^2 + 5   c_1^3 c_3 + 10   c_2^3 + 11   c_1 c_2 c_3  \label{lm} \\
&&\qquad{}- 5   c_1^2 c_4 -  c_3^2 - 
9   c_2 c_4 - 2  c_1 c_5 + 2   c_6)
\nonumber\\
&&{}+60480,\nonumber\end{eqnarray}
with $ 60480=2^6 \cdot 3^3 \cdot 5 \cdot 7$.

   \medskip
\noindent{\bf Case $7\mid c_1   $.} We write $c_1=7d_1$. Equation \eqref{77} divided by 7  gives
$$     - 10 d_1 c_2^3   + d_1 c_3^2 + 9 d_1 c_2 c_4   + 2^6\cdot 269 \equiv 0\pmod7, \text{ hence }
 d_1(3 c_2^3-c_3^2-2  c_2c_4 )\equiv    3  \pmod7,$$
 and, from  \eqref{lm} with $m=1$, we get  
$$24+2(10 c_2^3 - c_3^2 - 9 c_2 c_4  +2 c_6)\equiv 0\pmod7, \text{ hence }5+3 c_2^3 -c_3^2 - 2 c_2 c_4  +2 c_6  \equiv 0\pmod7.
$$
Together, these two congruences imply $  5 d_1+3    +2  d_1c_6 \equiv 0\pmod7$. Since $ d_1 c_6=224/7=32$, we finally get  $  d_1\equiv 2\pmod7$. But $d_1\le 1$ by \eqref{koo} and $ d_1\mid 2^5$, and all these conditions are incompatible.

 \medskip
 Since $c_1$ must be even, we now would like to exclude the cases $c_1\in \{\pm 2,\pm 4,-8,-16,-32\}$. Unfortunately, playing around with congruences is not enough when $c_1\in \{\pm 2,\pm 4\}$ (even with \eqref{tholt}) and we were unable to exclude these cases.
 
 We therefore assume $c_1\in \{-8,-16,-32\}$ and write $c_1=4 d_1 $.
%

    \medskip
Congruence \eqref{cog7} implies   $ c_2  \equiv 0\pmod{4}$ and we  write $c_2=4  d_2$. Equation \eqref{lm} with $m=1$, taken modulo 32 and divided by 2, gives
  \begin{eqnarray}
0&\equiv&    - 2  d_1 c_3^2   -  8 d_1 d_2 c_4 +  8 d_1 c_3 -  c_3^2 -   20 d_1 c_4 -   4  d_2 c_4 - 8  d_1 c_5
  +  8 d_1 + 8 d_2 \nonumber \\
  &&{}-14 c_4 + 2 c_6 + 12 \pmod{16},
\label{eq9}\end{eqnarray}
hence $c_3$ is even; we write $c_3=2 d_3$. When $c_1=\pm 32$, we already get the  contradiction $0\equiv 2^6\pmod{2^7}$  by reducing equation \eqref{77} modulo $2^7$.

So we assume $(d_1,c_6)\in \{(-2,-28),(-4,-14)\}$ and we write $c_6=2 d_6$. We obtain from \eqref{eq9} that $c_4$ is even, we write $c_4=2 d_4$, and, after dividing by 4, we get
  \begin{eqnarray*}
0&\equiv&    - 2  d_1 d_3^2     -  d_3^2 -  2 d_1 d_4 - 2 d_2  d_4 - 2  d_1 c_5
  +  2  d_1 +  2  d_2 - 7d_4 +  d_6 + 3 \pmod{4}
\end{eqnarray*}
hence $ d_3\equiv  d_4+ d_6+1\pmod2$.

   \medskip
\noindent{\bf Case     $  c_1=-16$. } We have  
$ d_3\equiv  d_4 \pmod2$. Substituting the integers $d_2$, $d_3$, and $d_4$ into \eqref{77} and dividing by 64, we find
$$160  d_2^3 - 10240  d_2^2 - 352  d_2  d_3 -  d_3^2 - 18  d_2  d_4 + 131072  d_2 +
4096  d_3 + 256  d_4 + 8 c_5 + 1883=0,$$
 hence $ d_3\equiv   1\pmod2$. Doing the same with \eqref{75}, we obtain, after   dividing by 4,  
$$96  d_2 d_4 + 3 d_3 d_4 - 3 d_2   c_5 - 2048  d_4 - 64 c_5 + 1932=0,$$
 hence $ d_2\equiv c_5\equiv  1\pmod2$. So we write $ d_2=2e_2+1$, $ d_3=2e_3+1$, $ d_4=2e_4+1$, $  c_5=2d_5+1$,  and, substituting these variables into \eqref{lm} with $m=-2$, we obtain the contradiction
 $$756 \chi(X,L^{-2})
 =
 -640 e_2^3 + 7424 e_2^2 + 704 e_2 e_3 + 2 e_3^2 + 36 e_2 e_4 + 80818 e_2 +
290 e_3 + 14 e_4 - 8 d_5 - 713529 $$
(the right side is odd, but the left side is even).

   \medskip
\noindent{\bf Case     $  c_1=-8$. } We have  
$ d_3\equiv d_4 +1\pmod2$. Substituting the integers $d_2$, $d_3$, and $d_4$ into \eqref{77} and dividing by 32, we find
$$160  d_2^3 - 2560  d_2^2 - 176  d_2  d_3 -  d_3^2 - 18  d_2  d_4 + 8192  d_2 + 512  d_3 +
64  d_4 + 4 c_5 + 3766=0,$$
hence $ d_3\equiv 0\pmod2$ and $ d_4\equiv 1\pmod2$. We write $ d_3=2e_3$ and $ d_4=2e_4+1$. Doing the same with   \eqref{lm} with $m=-1$, we get
  \begin{eqnarray*}
  	60480\chi(X,L^{-1})&=&
-50  d_2^3 + 1310  d_2^2 + 880  d_2  e_3 + 20  e_3^2 + 90 d_2 e_4+ 45 d_2  + 22670 d_2 \\
&&\hskip55mm {} - 160 e_3
- 20  e_4 - 80 c_5 - 469710,
  \end{eqnarray*}
  hence $d_2\equiv  0\pmod 2$. We write $ d_2=2e_2$  and, substituting  into \eqref{77} and dividing by $64$, we obtain the contradiction
 $$  	
 640 e_2^3 - 5120 e_2^2 - 352 e_2 e_3 - 2 e_3^2 - 36 e_2 e_4 + 8174 e_2 + 512 e_3
+ 64 e_4 + 2 c_5 + 1915
=0
.$$
This finishes the proof of the theorem.
 \end{proof}

  \section{Computations}\label{seco}
  \allowdisplaybreaks
  
  We list here the   polynomials $t_n(z;c_1,\dots,c_n)$ (defined in \eqref{tn})  for $n\le 9$ (they were given for $n\le 6$ in \cite[p.\ 145]{lw})
  \begin{eqnarray*}
t_1\hskip-3mm&=&\hskip-3mm c_1-\tfrac{1}{2} c_1z\\
t_2\hskip-3mm&=&\hskip-3mm c_2 -c_2z+\tfrac{1}{12}( c_1^2 + c_2)z^2\\
 t_3\hskip-3mm&=&\hskip-3mm  c_3 -\tfrac{3}{2} c_3 z+\tfrac{1}{12}(c_1 c_2 + 6 c_3)z^2-
 \tfrac{1}{2 4} c_1 c_2 z^3\\
 t_4\hskip-3mm&=&\hskip-3mm   c_4 -2  c_4z+\tfrac{1}{12}(c_1 c_3 + 14  c_4)z^2+\tfrac{1}{12}(c_1 c_3 + 2 c_4)z^3+\tfrac{1}{720} ( -c_1^4 + 4c_1^2 c_2+ 3 c_2^2 +  c_1 c_3-  c_4)z^4\\
 t_5\hskip-3mm&=&\hskip-3mm  c_5
 -\tfrac{5}{2}  c_5z
+\tfrac{1}{12} ( c_1  c_4 + 25  c_5)z^2
-\tfrac{1}{8}(  c_1  c_4 +5  c_5)z^3
\\
&& {}+\tfrac{1}{720} ( -c_1^3  c_2 +3 c_1  c_2^2 +c_1^2  c_3 + 29 c_1  c_4 +30  c_5)z^4+\tfrac{1}{1440} ( c_1^3  c_2 - 3 c_1  c_2^2 -  c_1^2  c_3 + c_1  c_4)z^5\\
  t_6\hskip-3mm&=&\hskip-3mm c_6
-3  c_6z
+\tfrac{1}{12} ( c_1  c_5 + 39  c_6)z^2
-\tfrac{1}{6} ( c_1  c_5 +9  c_6)z^3\\
&&{}
+\tfrac{1}{720} ( -c_1^3  c_3 +3 c_1  c_2  c_3 + c_1^2  c_4 - 3 c_3^2 +3 c_2  c_4 + 69  c_1  c_5 +186  c_6)z^4\\
&&{}
+\tfrac{1}{720}  (c_1^3  c_3 - 3c_1  c_2  c_3 -   c_1^2  c_4 + 3  c_3^2 - 3 c_2  c_4 - 9 c_1  c_5 - 6  c_6)z^5\\
&&{}
+\tfrac{1}{60480} ( 2c_1^6 - 12 c_1^4  c_2 +  11 c_1^2  c_2^2 +5 c_1^3  c_3 +   c_2^3 +  11 c_1  c_2  c_3 - 5 c_1^2  c_4\\
&&\hskip15mm {}-  c_3^2 - 9 c_2  c_4 - 2 c_1  c_5 + 2 c_6)z^6\\
   t_7\hskip-3mm&=&\hskip-3mm c_7
 -\tfrac{7}{2}  c_7z
+\tfrac{1}{12}  (c_1  c_6 + 56  c_7)z^2
-\tfrac{5}{24}(  c_1  c_6 + 14 c_7)z^3\\
&&{}
+\tfrac{1}{720}  (-c_1^3  c_4 + 3  c_1  c_2  c_4 + c_1^2  c_5- 3 c_3  c_4+3 c_2  c_5 +124  c_1  c_6+602 c_7)z^4\\
&&{}
+\tfrac{1}{480} ( c_1^3  c_4 -3 c_1  c_2  c_4 -  c_1^2  c_5 +3 c_3  c_4 - 3c_2  c_5 - 24 c_1  c_6 - 42  c_7)z^5\\
&&{}
+\tfrac{1}{60480}(2 c_1^5  c_2 - 10  c_1^3  c_2^2 - 2c_1^4  c_3 + 10 c_1  c_2^3 
+  11 c_1^2  c_2  c_3 - 40 c_1^3  c_4 -   c_1  c_3^2 \\
&&\hskip15mm {}+ 117 c_1  c_2  c_4 + 40  c_1^2  c_5 - 126  c_3  c_4 + 126c_2  c_5 + 170 c_1  c_6 + 84 c_7)z^6\\
&&{}
+\tfrac{1}{120960} ( - 2 c_1^5  c_2 + 10 c_1^3  c_2^2 +   2 c_1^4  c_3 -10 c_1  c_2^3 - 11 c_1^2  c_2  c_3- 2 c_1^3  c_4  \\
&&\hskip15mm {}+c_1  c_3^2+9c_1  c_2  c_4 +2c_1^2  c_5 -   2c_1  c_6)z^7\\
  t_8\hskip-3mm&=&\hskip-3mm c_8
-4  c_8z
+\tfrac{1}{12}  (c_1  c_7 + 76  c_8)z^2
-\tfrac{1}{4} (  c_1  c_7 + 20 c_8)z^3\\
&&{}
+\tfrac{1}{720}  (-c_1^3  c_5 + 3 c_1  c_2  c_5+  c_1^2  c_6 - 3 c_3  c_5+3 c_2  c_6+194  c_1  c_7+ 1458 c_8)z^4\\
&&{}
 +\tfrac{1}{360}(  c_1^3  c_5 - 3 c_1  c_2  c_5 -   c_1^2  c_6 + 3 c_3  c_5 - 3c_2  c_6 - 44  c_1  c_7 - 138  c_8)z^5\\
&&{}
 +\tfrac{1}{6048 0}(2 c_1^5  c_3 -10  c_1^3  c_2  c_3 -2  c_1^4  c_4 + 10 c_1  c_2^2  c_3 +10  c_1^2  c_3^2 +  c_1^2  c_2  c_4 - 96 c_1^3  c_5 - 10 c_2  c_3^2 
\\
&&  {} + 10 c_2^2  c_4 - 11  c_1  c_3  c_4 + 295 c_1  c_2  c_5 + 96 c_1^2  c_6 - 20  c_4^2 -  264c_3  c_5 + 284 c_2  c_6 \\
&&\hskip10mm {}+1206 c_1  c_7 + 1524 c_8)z^6\\
&&{}
+\tfrac{1}{60480} (-2 c_1^5  c_3 + 10 c_1^3  c_2  c_3 + 2 c_1^4  c_4 - 10 c_1  c_2^2  c_3 - 10 c_1^2  c_3^2 -   c_1^2  c_2  c_4 + 12 c_1^3  c_5 + 10 c_2  c_3^2  \\
&&  {}- 10 c_2^2  c_4+  11 c_1  c_3  c_4 - 43 c_1  c_2  c_5 - 12c_1^2  c_6 + 20 c_4^2 + 12 c_3  c_5 - 32 c_2  c_6\\
&&\hskip10mm {} -30 c_1  c_7 - 12 c_8)z^7\\
&&{}
 +\tfrac{1}{3628800}(-3  c_1^8 + 24 c_1^6  c_2 - 50  c_1^4  c_2^2 - 14 c_1^5  c_3 + 8  c_1^2  c_2^3 + 26 c_1^3  c_2  c_3 + 14 c_1^4  c_4 +21 c_2^4 
 \\
 &&\hskip10mm {}+ 50  c_1  c_2^2  c_3 + 3 c_1^2  c_3^2 - 19 c_1^2  c_2  c_4 - 7  c_1^3  c_5 - 8 c_2  c_3^2 -34  c_2^2  c_4 - 13 c_1  c_3  c_4  \\
&&\hskip10mm {}- 16 c_1  c_2  c_5 + 7 c_1^2  c_6 +5 c_4^2 + 3 c_3  c_5 + 13  c_2  c_6 + 3 c_1  c_7 -3 c_8)z^8\\
    t_9\hskip-3mm&=&\hskip-3mm c_9
  -\tfrac{9}{2}  c_9z
 +\tfrac{1}{12} ( c_1  c_8 + 99  c_9)z^2
+\tfrac{1}{24} (-7 c_1  c_8 -189  c_9)z^3\\
 &&\hskip-3mm{}+\tfrac{1}{720} (- c_1^3  c_6 + 3  c_1  c_2  c_6 +  c_1^2  c_7 - 3 c_3  c_6 + 3 c_2  c_7 + 279 c_1  c_8 + 2979c_9)z^4\\
 &&\hskip-3mm{}
 +\tfrac{1}{2 88} ( c_1^3  c_6 - 3  c_1  c_2  c_6 - c_1^2  c_7 +3 c_3  c_6 - 3 c_2  c_7 -69 c_1  c_8 - 333 c_9)z^5\\
 &&\hskip-3mm{}
 +\tfrac{1}{60480} (2 c_1^5  c_4 - 10  c_1^3  c_2  c_4 -2  c_1^4  c_5 + 10 c_1  c_2^2  c_4 + 10c_1^2  c_3  c_4 +   c_1^2  c_2  c_5 - 173 c_1^3  c_6 - 10c_2  c_3  c_4 
 \\
&&\hskip1cm {}
 - 10c_1  c_4^2 + 10c_2^2  c_5 -    c_1  c_3  c_5 + 526 c_1  c_2  c_6 
  \\
&&\hskip1cm {}+ 173  c_1^2  c_7 - 10 c_4  c_5 -  505 c_3  c_6 + 515 c_2  c_7 
 + 4041 c_1  c_8 + 9075c_9)z^6
\\
 &&\hskip-3mm{}+\tfrac{1}{40320} (-2 c_1^5  c_4 + 10  c_1^3  c_2  c_4 + 2 c_1^4  c_5 - 10 c_1  c_2^2  c_4 - 10 c_1^2  c_3  c_4 -    c_1^2  c_2  c_5 + 33  c_1^3  c_6 +10c_2  c_3  c_4 
 \\
&&\hskip1cm {} +10 c_1  c_4^2 -10 c_2^2  c_5 +  c_1  c_3  c_5 - 106  c_1  c_2  c_6 - 33c_1^2  c_7 + 10  c_4  c_5 + 85 c_3  c_6 - 95  c_2  c_7  \\
&&\hskip1cm {}- 261  c_1  c_8 - 255 c_9)z^7\\
 &&\hskip-3mm{}
 + \tfrac{1}{3628800}(-3 c_1^7  c_2 +21 c_1^5  c_2^2 +3c_1^6  c_3 - 42 c_1^3  c_2^3 - 29 c_1^4  c_2  c_3 
 \\
&&\hskip1cm {} + 57c_1^5  c_4 + 21c_1  c_2^4 +50  c_1^2  c_2^2  c_3 + 8  c_1^3  c_3^2 - 274c_1^3  c_2  c_4 - 57 c_1^4  c_5 - 8  c_1  c_2  c_3^2  \\
&&\hskip1cm {}+ 266  c_1  c_2^2  c_4 +287  c_1^2  c_3  c_4 + 14 c_1^2  c_2  c_5 - 153 c_1^3  c_6 - 300 c_2  c_3  c_4 - 295  c_1  c_4^2   \\
&&\hskip1cm {}+ 300 c_2^2  c_5- 27 c_1  c_3  c_5 + 673 c_1  c_2  c_6 + 153 c_1^2  c_7 - 300 c_4  c_5 - 30c_3  c_6 +  330 c_2  c_7  \\
&&\hskip1cm {}+ 267  c_1  c_8+90 c_9)z^8
 \\
 &&\hskip-3mm{}+\tfrac{1 }{7257600}(3c_1^7  c_2 - 21 c_1^5  c_2^2 - 3 c_1^6  c_3 + 42 c_1^3  c_2^3 + 29 c_1^4  c_2  c_3 +3 c_1^5  c_4  - 21c_1  c_2^4 \\
&&\hskip1cm {}- 50 c_1^2  c_2^2  c_3 - 8 c_1^3  c_3^2 - 26  c_1^3  c_2  c_4 - 3  c_1^4  c_5 +8c_1  c_2  c_3^2 + 34 c_1  c_2^2  c_4 \\
&&\hskip1cm {}+ 13c_1^2  c_3  c_4 +16 c_1^2  c_2  c_5 + 3 c_1^3  c_6 -5 c_1  c_4^2 - 3  c_1  c_3  c_5 -13  c_1  c_2  c_6 \\
&&\hskip1cm {}-3 c_1^2  c_7 + 3 c_1  c_8)z^9
\end{eqnarray*}

 \end{document}